\documentclass{amsart}
\usepackage{amssymb}
\usepackage{mathtools}
\usepackage{microtype}
\usepackage[svgnames]{xcolor}
\usepackage[unicode,
  colorlinks=true,
  linktocpage=true,
  citecolor=OliveDrab,
  linkcolor=DarkMagenta,
  linkcolor=DarkMagenta,
  pdftitle={Binary partial groups},
  pdfauthor={Philip Hackney, Justin Lynd, and Edoardo Salati}
  ]{hyperref}
\usepackage{tikz-cd}
\usepackage{enumitem}
\usepackage[capitalise]{cleveref}

\newcommand{\fin}{\Upsilon}
\newcommand{\ord}{\Delta}
\DeclareMathOperator{\id}{id}
\newcommand{\op}{\textup{op}}
\newcommand{\set}{\mathsf{Set}}
\newcommand{\sset}{\mathsf{sSet}}
\newcommand{\keq}{\asymp}
\DeclareMathOperator{\sk}{sk}

\usepackage{mathrsfs}
\newcommand{\edgemap}{\mathscr{E}}

\newtheorem{theorem}{Theorem}
\newtheorem{proposition}[theorem]{Proposition}
\newtheorem{lemma}[theorem]{Lemma}
\newtheorem{corollary}[theorem]{Corollary}

\theoremstyle{definition}
\newtheorem{definition}[theorem]{Definition}
\newtheorem{example}[theorem]{Example}

\theoremstyle{remark}
\newtheorem{remark}[theorem]{Remark}

\AddToHook{env/proposition/begin}{\crefalias{theorem}{proposition}}
\AddToHook{env/corollary/begin}{\crefalias{theorem}{corollary}}
\AddToHook{env/lemma/begin}{\crefalias{theorem}{lemma}}
\AddToHook{env/definition/begin}{\crefalias{theorem}{definition}}
\AddToHook{env/example/begin}{\crefalias{theorem}{example}}
\AddToHook{env/remark/begin}{\crefalias{theorem}{remark}}


\begin{document}

\title{Binary partial groups}

\author{Philip Hackney}
\address{Department of Mathematics, University of Louisiana at Lafayette}
\email{philip@phck.net} 
\urladdr{http://phck.net}

\author{Justin Lynd}
\address{Department of Mathematics, University of Louisiana at Lafayette}
\email{lynd@louisiana.edu}

\author{Edoardo Salati}
\address{Department of Mathematics, RPTU University Kaiserslautern-Landau}
\email{edoardo.salati@rptu.de}

\date{July 7, 2026}

\thanks{
This work was supported by a grant from the Simons Foundation (\#850849, PH)
and by a grant from the Simons Foundation International (SFI-MPS-TSM-00014188, JL). 
PH was partially supported Louisiana Board of Regents through the Board of Regents Support fund LEQSF(2024-27)-RD-A-31.
}

\begin{abstract}
There are many examples of `binary' partial groups in the literature: sets equipped an identity and a partially-defined binary operation, such that each element admits an inverse.
We show that many of these may be regarded as partial groups in the sense of Chermak, and single out the largest class of such objects.
\end{abstract}

\maketitle

The idea of `partial groups' is nothing new, going back at least three-quarters of a century.
But what do these legacy partial groups have to do with the partial groups introduced by Chermak \cite{Chermak:FSL}? 
In his setup, the choice of set of multipliable words (of arbitrary length) is an essential and subtle part of the structure, not simply derived from the partial binary multiplication.
Lemoine and Molinier gave one answer by showing that Stallings' pregroups comprise a full subcategory of Chermak's partial groups \cite{LemoineMolinier:PGPRFS}.
We go further, showing how to regard a large number of  structures weaker than pregroups as Chermak partial groups.
In fact, we identify the largest class of classical partial groups where this is possible --- the \emph{binary partial groups} (\cref{def bin partial group}), which were studied by Baer and Tamari.

We provide two canonical embeddings from the category of binary partial groups to the category of partial groups, corresponding to the minimal and maximal way to encode a binary partial group as a Chermak partial group.
The small embedding (using the minimal encoding) is shown to induce an equivalence between binary partial groups and a known subcategory, the `2-skeletal partial groups.'
The interpretation of this subcategory relies on viewing a partial group as a special kind of symmetric set, as in \cite{HackneyLynd:PGSSS}; $n$-skeletal partial groups were studied in \cite{HackneyMolinier:DPG} (a specialization of $n$-skeletal symmetric sets \cite{Antokoletz:NAMCHT,HoraKamioMaehara}) and include all partial groups with at most $n$ non-identity elements.

We assume familiarity with simplicial sets and categories.
If $a$ and $b$ are two expressions, which may or may not be defined, we write $a\keq b$ to mean either both $a$ and $b$ are undefined, or both are defined and $a=b$.

\section{Chermak partial groups}\label{sec chermak partial groups}
In \cite{Chermak:FSL}, Chermak introduced a notion of partial group in order to study certain problems in $p$-local group theory. 
The basic definitions may be recast either in the setting of simplicial sets \cite{BrotoGonzalez:ETPG,gonzalez} or in the setting of symmetric sets \cite{HackneyLynd:PGSSS}.

A simplicial set $X \colon \ord^\op \to \set$ is said to be \emph{edgy} if the Segal maps $\edgemap_n \colon X_n \to X_1^{\times n}$ (whose components come from iterated outer face maps) 
are injective for $n\geq 2$.
This structure endows $X_1$ with partially-defined $n$-to-$1$ operations for $n\geq 2$ via the following span
\[ \begin{tikzcd}
X_1^{\times n} & X_n \lar[hook',"\edgemap_n"'] \rar & X_1,
\end{tikzcd} \]
whose the rightmost leg is the composite of inner face maps.
In particular, monoids, groups, and categories all form edgy simplicial sets.

We let $\fin$ denote the category whose objects are the the same as $\ord$ (that is, the sets $[n] = \{0, 1, \dots, n\}$ for $n\geq 0$) and whose morphisms are arbitrary functions.
A symmetric (simplicial) set is a presheaf on $\fin$, i.e.\ a functor $\fin^\op \to \set$.
Restricting along the inclusion $\ord^\op \to \fin^\op$ gives a forgetful functor from symmetric sets to simplicial sets.
A symmetric set $X$ is said to be \emph{spiny} if its underlying simplicial set is edgy (other characterizations are given in \cite[\S3]{HackneyLynd:PGSSS}).
A \emph{partial group} is a spiny symmetric set $X$ which is furthermore reduced (i.e.\ $X_0$ is a one-point set).
The category of partial groups is a full subcategory of (spiny) symmetric sets.

Spiny symmetric sets may instead be fully described in terms of edgy simplicial sets \cite[Theorem 4.6]{HackneyLynd:PGSSS}. 
The forgetful functor from spiny symmetric sets to $\sset$ is fully faithful, and an edgy simplicial set $X$ is in the essential image if and only if it admits a simplicial set map $\dagger \colon X^\op \to X$ such that
\begin{enumerate}
\item $\dagger \dagger^\op = \id_X$, and 
\item if $x\in X_n$ satisfies $\edgemap_n(x) = (a_1,\dots, a_n)$, then there is an element $y \in X_{2n}$ such that $\edgemap_{2n}(y) = (a_n^\dagger, \dots, a_1^\dagger, a_1,\dots, a_n)$ and the total product of $y$ is a degenerate 1-simplex.
\end{enumerate}
In \cite{BrotoGonzalez:ETPG,gonzalez} it is shown that the category of \emph{reduced} edgy simplicial sets satisfying the preceding conditions is equivalent to Chermak's category of partial groups from \cite{Chermak:FSL}. 

\section{Binary partial groups}

A unital partial magma\footnote{A (nonempty) set equipped with a partially-defined binary operation has gone by a number of names, including add \cite{MR30953}, pree \cite{MR971352}, monoid \cite{MR393288}, partial groupoid \cite{Evseev:SPG} or pargoid \cite{MR1465438}, halfgroupoid \cite{MR93552}, and partial magma \cite{MR1727714}.
In the present decade folks seem to be using the term `partial magma,' and we do the same.} is a set $P$ equipped with a partially defined binary operation $P\times P \nrightarrow P$ and an element $1$ such that $a1\keq a \keq 1a$ for all $a\in P$.
The study of unital partial magmas where each element admits an inverse goes back at least to Baer \cite{MR30953,MR38974,MR38975}.

\begin{definition}\label{def bin partial group}
A \emph{binary partial group} is a unital partial magma $P$ such that for each $a\in P$ there exists $a^\dagger \in P$ satisfying both of the following for all $b\in P$:
\begin{enumerate}
		\item If $ab$ is defined then $a^\dagger(ab) \keq b$. \label{bpg left}
		\item If $ba$ is defined then $(ba)a^\dagger \keq b$. \label{bpg right}
\end{enumerate}
\end{definition}

This is simply a convenient way to phrase the definition, but we could also separate \eqref{bpg left} and \eqref{bpg right} for example, or we could insist that $a^\dagger$ is unique.
Assume there is an $a^\dagger$ such that \eqref{bpg left} holds and an $a^\ddagger$ such that \eqref{bpg right} holds.
Then \eqref{bpg right} gives $a a^\ddagger = (1a)a^\ddagger = 1$, while \eqref{bpg left} gives $a^\dagger = a^\dagger 1 = a^\dagger(aa^\ddagger) = a^\ddagger$.
Thus we were required to use the same element in both parts.
Moreover, a binary partial group possesses a \emph{unique} function $\dagger \colon P \to P$ making the conditions hold.

\begin{remark}
A partial magma Baer called `self-reflexive in the strict sense' \cite[\S4]{MR30953} is a binary partial group; Tamari's term was `symmetrical monoid' \cite{Tamari:PAMPMG,MR393288}.
\end{remark}

The following was known to Tamari \cite[p.14]{Tamari:PAMPMG}, but we include a proof as the one we've seen in the literature \cite[1.2]{MR938449} uses associativity axiom \eqref{assoc ternary} below.

\begin{lemma}\label{lem anti-auto}
In a binary partial group, $\dagger$ is an involutive anti-isomorphism, i.e. $(a^\dagger)^\dagger = a$ and $(ab)^\dagger \keq b^\dagger a^\dagger$.
\end{lemma}
\begin{proof}
Involutivity is as one would expect: $a = a^{\dagger\dagger} (a^\dagger a) = a^{\dagger \dagger}(1) = a^{\dagger \dagger}$.
Assume $pq$ is defined.
By \eqref{bpg left} we have $p^\dagger(pq) = q$.
Apply \eqref{bpg right} with $b= p^\dagger$ and $a = pq$ to see $q(pq)^\dagger = (p^\dagger(pq))(pq)^\dagger = p^\dagger$.
Apply \eqref{bpg left} with $a=q$ and $b = (pq)^\dagger$ to see $ (pq)^\dagger = q^\dagger (q (pq)^\dagger) = q^\dagger p^\dagger$.
Thus $ab$ defined implies $(ab)^\dagger = b^\dagger a^\dagger$.
If $b^\dagger a^\dagger$ is defined, we then have 
$(b^\dagger a^\dagger)^\dagger = (a^\dagger)^\dagger (b^\dagger)^\dagger = ab$, 
which suffices for the other direction.
\end{proof}

Stallings' pregroups \cite{MR255689,Stallings:GT3DM} are a special kind of binary partial group, as are a wide variety of subsequent generalizations (see \cite[\S4]{MR1826094} for an overview\footnote{These generalized Stallings' pregroups all have additional axioms beyond \eqref{assoc ternary} and \eqref{two-sided inverse} which guarantee embeddability in a group.}).
The common generalization of all of these is a unital partial magma satisfying two conditions.
The first is a reasonable associativity condition for a partial magma:
\begin{enumerate}[label=\texttt{A\arabic*}., ref=\texttt{A\arabic*}, start=3]
\item If $ab$ and $bc$ are defined, then $(ab)c \keq a(bc)$. 
\label{assoc ternary}
\end{enumerate}
The second is a two-sided inverse property.
\begin{enumerate}[label=\texttt{I\arabic*}., ref=\texttt{I\arabic*}, start=2]
\item For each $a\in P$ there exists an element $a^\dagger $ such that $aa^\dagger \keq 1 \keq a^\dagger a$. \label{two-sided inverse}
\end{enumerate}
Baer showed that if $P$ is a unital partial magma satisfying \eqref{assoc ternary} and where each element admits a right inverse, then $P$ is a binary partial group \cite[p.627]{MR38974}.
A binary partial group need not satisfy \eqref{assoc ternary}, though of course it will always satisfy \eqref{two-sided inverse}.

\begin{example}
The `regular partial groups' from \cite{Jekel:PGEA} are binary partial groups as \eqref{two-sided inverse} is an axiom, and regularity guarantees \eqref{assoc ternary}.
The `partial groups' from \cite{MR188332} are pregroups, hence binary partial groups.
\end{example}

\section{The big embedding of binary partial groups}

Let $P$ be a binary partial group.
Define a simplicial set $BP$ by letting $BP_n \subseteq P^{\times n}$ be the set of length $n$ words $w$ such that 1) each full parenthesization of $w$ gives a valid iterated multiplication in $P$, and 2) the values of all of these multiplications are the same.
In particular, $BP_0 = \ast$, $BP_1 = P$, and $BP_2 \subseteq P\times P$ is the set of multipliable pairs.
It is clear that $BP$ possesses inner face maps (using multiplication in $P$ at some position) and degeneracy maps (by placing the unit into some position), and that these again land in $BP$.
It also has outer face maps by tossing away the first or last elements; this function landing in $BP$ depends on the cancellation property.

\begin{remark}
In the special case when $P$ is a pregroup in the sense of Stallings, this construction agrees with the partial group from \cite[\S2.3]{LemoineMolinier:PGPRFS}.
Their description looks a bit different since they make reference to the fact that $P$ embeds in a group, which is not generally true for binary partial groups.
\end{remark}

If $w = (a_1, \dots, a_n) \in P^{\times n}$, we write $w^\dagger = (a_n^\dagger, \dots, a_1^\dagger) \in P^{\times n}$.
Our next lemma shows $w\in BP_n$ if and only if $w^\dagger \in BP_n$.
For this we need a minor preliminary.
If $\mu \colon P^{\times n} \nrightarrow P$ is a partial function arising from multiplication specified by a full parenthesization of $n$ letters, then we write $\bar \mu \colon P^{\times n} \nrightarrow P$ for the partial function arising from the mirrored parenthesization, 
e.g.\ $a(b(cd))$ and $((ab)c)d$ are mirrors, as are 
$(a((bc)d))(ef)$ and $(ab)((c(de))f)$.
Inductively we have that if $\mu(v)$ is defined, then $\mu (v) = \bar\mu(v^\dagger)^\dagger$. 
This is clear if $v$ has length 1, since $\mu (v) = v = \bar\mu(v)$.
For longer words, splitting at the last multiplication we have $\mu (v) \keq \mu_1(v_1) \mu_2(v_2)$ where $v_1, v_2$ are words of (shorter) positive length, $v=v_1v_2$ is their concatenation, and $\mu_1$, $\mu_2$ are already known to satisfy the condition.
But then
\[
	\mu (v) = \mu_1(v_1) \mu_2(v_2) = \bar\mu_1(v_1^\dagger)^\dagger \bar\mu_2(v_2^\dagger)^\dagger = \left[ \bar\mu_2(v_2^\dagger) \bar\mu_1(v_1^\dagger) \right]^\dagger = \left[ \bar\mu(v_2^\dagger v_1^\dagger) \right]^\dagger = \bar\mu(v^\dagger)^\dagger.
\] 

\begin{lemma}\label{inversion closure}
If $w\in BP_n$, then $w^\dagger \in BP_n$.
In this case, if $\mu(w) = a$ then $\mu(w^\dagger) = a^\dagger$.
\end{lemma}
\begin{proof}
Let $v = w^\dagger$, and $\mu_1$, $\mu_2$ two parenthesizations of $n$ letters.
By the above consideration, and since $\bar \mu_1(w) = \bar \mu_2(w)$ as $w\in BP_n$, we have
\[
	\mu_1 (v) \keq \bar \mu_1(w)^\dagger \keq \bar \mu_2(w)^\dagger \keq \mu_2(v),
\]
so $v\in BP_n$.
\end{proof}

\begin{theorem}
If $w\in BP_n$ then, then $w^\dagger w \in BP_{2n}$.
\end{theorem}
\begin{proof}
Using double induction on $n$ and $0 \leq k \leq n$, we prove the following statement: if $w = (a_1, \dots, a_n) \in BP_n$, then $w^{(k)} = (a_k^\dagger, \dots, a_1^\dagger, a_1, \dots, a_n)$ is in $BP_{n+k}$ and has total product equal to $a_{k+1}\cdots a_n$.
Notice that this statement implies, by \cref{inversion closure}, that $(a_n^\dagger, \dots, a_1^\dagger, a_1, \dots, a_k) \in BP_{n+k}$ with total product $(a_{k+1}\cdots a_n)^\dagger$.
It is clear that the statement is true when $n=1$ or $k=0$.

Fix $n > 1$ and $1 \leq k \leq n$ such that the statement holds for words of length $1 \leq m < n$, and for $w^{(h)}$ when $w$ is a length $n$ word and $0 \leq h < k$.
We break $w^{(k)}$ into two positive length subwords in every possible way, and notice by induction that each side has a unique product.
For $0\leq j < k$ we have (by \eqref{bpg left} in \cref{def bin partial group})
\begin{align*}
	(a_k^\dagger \cdots a_{j+1}^\dagger)(a_j^\dagger \cdots a_1^\dagger a_1 \cdots a_n) &= (a_{j+1}\cdots a_k)^\dagger (a_{j+1} \cdots a_n)\\
	&= (a_{j+1}\cdots a_k)^\dagger ((a_{j+1} \cdots a_k)(a_{k+1} \cdots a_n)) 
\\
	&= (a_{k+1} \cdots a_n),
\end{align*}
for $1\leq i < k$ we have
\begin{align*}
	(a_k^\dagger \cdots a_1^\dagger a_1 \cdots a_i)(a_{i+1}\cdots a_n) &= (a_{i+1}\cdots a_k)^\dagger (a_{i+1} \cdots a_n)\\
	&= (a_{i+1}\cdots a_k)^\dagger ((a_{i+1} \cdots a_k)(a_{k+1} \cdots a_n)) 
\\
	&= (a_{k+1} \cdots a_n),
\end{align*}
and for $k \leq \ell < n$ we have
\[
(a_k^\dagger \cdots a_1^\dagger a_1 \cdots a_\ell)(a_{\ell+1}\cdots a_n) = (a_{k+1} \cdots a_\ell)(a_{\ell+1}\cdots a_n) = (a_{k+1} \cdots a_n).
\]
Thus every full parenthesization of $w^{(k)}$ is multipliable, and all have the same value.
The theorem statement is the special case $n=k$.
\end{proof}

As the total product of $w^\dagger w$ is $1$, we deduce from \S\ref{sec chermak partial groups}:

\begin{corollary}
$BP$ is a Chermak partial group. \qed
\end{corollary}

A map of binary partial groups is a (total) function respecting the partial binary multiplication.
That is, if $ab$ is defined then $f(ab) \keq f(a)f(b)$.
By cancellation, such a function automatically preserves the unit and inverses.
The assignment $P \mapsto BP$ is functorial: a map of binary partial groups $f \colon P \to Q$ induces the entrywise function $P^{\times n} \to Q^{\times n}$; to define the simplicial map $Bf \colon BP \to BQ$ it suffices to show that $BP_n \subseteq P^{\times n} \to Q^{\times n}$ lands in $BQ_n$.
This is the case, as for any multiplication $\mu \colon P^{\times n} \nrightarrow P$, preservation of products implies $f(\mu(w)) = \mu(f(w))$ ($w\in BP_n)$, so all multiplications of $f(w)$ exist and are equal.
Hence $f(w) \in BQ_n$.

As indicated in \S\ref{sec chermak partial groups}, given any edgy simplicial set $X$, there is a canonical partial magma structure on $X_1$ given by the span 
\[ X_1 \times X_1 \hookleftarrow X_2 \to X_1.\]
In other words, $xy$ is defined if and only if $(x,y) \in X_1 \times X_1$ is the spine of a 2-simplex $z\in X_2$, in which case $xy$ is defined to be the long edge of $z$ (i.e.\ $xy \coloneq d_1(z)$).
If $X$ is reduced, then $X_1$ is a unital partial magma.

\begin{proposition}
If $X$ is a Chermak partial group, then $X_1$ equipped with its  partial binary multiplication is a binary partial group.
\end{proposition}
\begin{proof}
If $a\in X_1$ then we have $(a^\dagger, a), (a,a^\dagger) \in X_2 \subseteq X_1\times X_1$, both having product $1$.
If $(a,b) \in X_2$ (i.e., $ab$ is defined) then $(b^\dagger, a^\dagger, a, b) \in X_4$, so $(a^\dagger, a, b) \in X_3$.
We then have 
\[
	a^\dagger(ab) = d_1 d_2(a^\dagger, a, b) = d_1 d_1 (a^\dagger, a, b) = d_1(a^\dagger a, b) = d_1(1,b) = b.
\]
A mirrored argument shows that if $ba$ is defined then $(ba)a^\dagger = b$.
\end{proof}

Write $T$ for the functor from Chermak partial groups to binary partial groups given in the proposition.
It is immediate that $TB = \id$.

\begin{theorem}\label{unit nat trans}
The functor $T$ is left adjoint to the functor $B$ from binary partial groups to Chermak partial groups, with unit $\eta \colon \id \Rightarrow BT$ given by the Segal maps.
\end{theorem}
\begin{proof}
Let $X$ be a Chermak partial group.
First notice that $\edgemap_n \colon X_n \hookrightarrow X_1^{\times n}$ lands in $BTX_n$.
Indeed, a multiplication map $\mu \colon X_1^{\times n} \to X_1$ is nothing but a composite of partial inner face maps, which are total on $\edgemap_n(X_n)$, and which all coincide by the simplicial identities in $X$.
A quick check gives that $X \to BTX$ is a map of simplicial sets, hence a map of Chermak partial groups.
Thus the Segal maps define a natural transformation $\eta$.
The counit $\varepsilon \colon TB \Rightarrow \id$ is the identity, so the triangle identities reduce to $T(\eta_X) = \id_{TX}$ and $\eta_{BP} = \id_{BP}$.
These both hold since $\eta$ is the identity in simplicial degree 1.
\end{proof}

Since the counit of the adjunction is an isomorphism, we have \cite[Lemma 4.5.13]{Riehl:CTC}:

\begin{corollary}
The functor $B$ from binary partial groups to Chermak partial groups is fully faithful. \qed
\end{corollary}

As a special case, we recover \cite[Proposition 2.23]{LemoineMolinier:PGPRFS} which states that the category of Stallings' pregroups embeds fully faithfully into that of Chermak partial groups.

\section{The small embedding of binary partial groups}

The defining multiplication in a binary partial group takes a \emph{pair} of elements and possibly produces their product, yet $BP$ is built by considering iterated versions of those products, in all possible orders.
As this is a bit complicated, we provide an alternative embedding which keeps the focus on pairs and constructs an object of dimension at most 2.

The 2-skeleton of a symmetric set $X$ is the smallest symmetric subset $\sk_2 X \subseteq X$ containing all of $X_0$, $X_1$, and $X_2$.
A symmetric set $X$ is called 2-skeletal if $\sk_2 X = X$.
The category of 2-skeletal symmetric sets is equivalent to the category of $\fin_{\leq 2}$-presheaves, where $\fin_{\leq 2} \subset \fin$ is the full subcategory on the objects $[0], [1], [2] \in \fin$.
The 2-skeleton of a (Chermak) partial group is again a partial group.
If $X = \sk_2 X$ is a 2-skeletal symmetric set, then $X$ is spiny if and only if $\edgemap_2 \colon X_2 \to X_1\times X_1$ is injective \cite[Theorem 9]{HackneyMolinier:DPG}.

Recall the underlying binary partial group functor $T$.
Precomposing with the endofunctor $\sk_2$ on Chermak partial groups, we have $T\circ \sk_2 = T$.

\begin{definition}
The \emph{small embedding} of binary partial groups into Chermak partial groups is the composite functor $B' = \sk_2 \circ B$.
\end{definition}

The functor $B'$ is left adjoint to $T$.
It induces the following equivalence of categories.

\begin{theorem}\label{thm 2skel binary}
The category of 2-skeletal Chermak partial groups is equivalent to the category of binary partial groups.
\end{theorem}
\begin{proof}
Notice $TB' = T\circ (\sk_2 \circ B) = (T\circ \sk_2) \circ B = TB = \id$.
The natural transformation $\eta \colon \id \Rightarrow BT$ from \cref{unit nat trans} whiskers to a natural transformation $\eta' \coloneq \sk_2 \circ \eta \colon \sk_2 \Rightarrow \sk_2 \circ (BT) = B'T$.
If $X$ is 2-skeletal, then $\eta'_X \colon X = \sk_2 X \to B'T X$ is a map which is given by the identity in level 1 and $\edgemap_2 \colon X_2 \to BTX_2 \subseteq X_1 \times X_1$ in level 2.
But $BTX_2$ consists of the multipliable pairs of elements in $TX$, which is precisely the image of $\edgemap_2 \colon X_2 \to X_1 \times X_1$.
Thus $\eta'_X$ is an isomorphism in levels 0, 1, and 2. 
As $\eta'_X$ is a map between 2-skeletal symmetric sets, it is an isomorphism.
\end{proof}

To summarize, we have the following diagram of categories and functors.
\[ \begin{tikzcd}[column sep=large]
& \text{2-skeletal partial groups} \dar[shift right=2, hook'] \dar[phantom, "\dashv"{scale=0.7}]
\\
\text{Binary partial groups} \rar[shift right=2, hook,"B"'] \urar[end anchor={west}, bend left=20, "B'=\sk_2 \circ B", "\simeq"'] \rar[phantom,"\perp"{scale=0.6}] 
& \text{Chermak partial groups} \lar[shift right=2, "T"'] \uar[shift right=2,"\sk_2"']
\end{tikzcd} \]
As a curiosity: for each $2 < n < \infty$, the functor $\sk_n \circ B$ is also a fully faithful embedding of binary partial groups into Chermak partial groups.
 
\begin{remark}
The chief use of symmetric sets in this paper is \cref{thm 2skel binary}, as the 2-skeleton we use is the \emph{symmetric} notion.
Indeed, there is only one 2-skeletal \emph{simplicial} set which is a Chermak partial group (the trivial group).
\end{remark}

\begin{remark}
Suppose $P$ is a partial magma and $X$ is a Chermak partial group such that $X_1 = P$ where the underlying partial binary operations agree.
Then $P$ is a binary partial group, since $P$ is equal to the partial magma $TX$.
\end{remark}

\bibliographystyle{amsplain}
\bibliography{binary_partial_groups}

\end{document}